\documentclass[12pt]{amsart}
\usepackage{amssymb,amsthm,amsmath,amstext,tikz-cd}
\usepackage{mathrsfs}  % allows nice script font for sheaves
\usepackage{bm}        % allows bold italic letters in math mode
\usepackage{mathtools} % allows more extendible arrows
\usepackage{changepage,enumitem}
\usepackage{color,fullpage}

\usepackage[all]{xy}
 \DeclareFontFamily{U}{wncy}{}
    \DeclareFontShape{U}{wncy}{m}{n}{<->wncyr10}{}
    \DeclareSymbolFont{mcy}{U}{wncy}{m}{n}
    \DeclareMathSymbol{\Sha}{\mathord}{mcy}{"58} 

\theoremstyle{plain}
\newtheorem{theorem}{Theorem}[section]
\newtheorem{prop}[theorem]{Proposition}
\newtheorem{lemma}[theorem]{Lemma}

\newtheorem{const}[theorem]{Construction}

\newtheorem*{theoremintro*}{Theorem}
\newtheorem*{maintheoremintro*}{Main Theorem}

\newtheorem*{propintro*}{Proposition}

\theoremstyle{definition}

\newtheorem{example}[theorem]{Example}

\newtheorem{remark}[theorem]{Remark}

\newtheorem*{ack}{Acknowledgments}

\newcommand{\ind}{\text{ind }}

\DeclareMathOperator{\Br}{\mathrm{Br}}

\usepackage[backref=page]{hyperref}
\hypersetup{pdftitle={Possible values of the m-invariant}}
\hypersetup{pdfauthor={Connor Cassady}}
\hypersetup{colorlinks=true,linkcolor=red,anchorcolor=blue,citecolor=blue}
\usepackage{cleveref}

\title[Odd $m$-invariants]{On possible values of the $m$-invariant}
\date{\today}
\author[Cassady]{Connor Cassady}

\begin{document}

\maketitle

\begin{abstract}
We show that every positive integer different from $3$ and $5$ can be realized as the $m$-invariant of a field.
\end{abstract}

\section*{Introduction}
Certain field invariants are useful in the study of quadratic forms over fields of characteristic unequal to $2$ since they give a measure of the complexity of quadratic forms over the field. One such invariant that has been well-studied is the $u$-\textit{invariant} of a field $k$, denoted by $u(k)$ and defined to be the maximal dimension of an anisotropic quadratic form over $k$ \cite[Definition~XI.6.1]{lam}. On the other hand, the $m$-\textit{invariant} of $k$, denoted by $m(k)$ and defined in \cite{m-inv} to be the minimal dimension of an anisotropic universal quadratic form over~$k$, has received much less attention thus far. In this article we are motivated by determining the possible values of the $m$-invariant. That is, given an integer $n \geq 1$, is there a field $k$ with $m(k) = n$?

The problem of determining the possible values of the $u$-invariant has been extensively studied. We know that $u(k) \ne 3, 5,$ or $7$ for any field $k$ of characteristic $\ne 2$ (see, e.g., \cite[Proposition~XI.6.8]{lam}), and for each integer $n \geq 0$ one can easily find fields with $u$-invariant~$2^n$ (e.g., $u(\mathbb{C}) = 1$ and $u(\mathbb{C}((t_1))((t_2)) \cdots ((t_n))) = 2^n$). In fact, Kaplansky conjectured that if $u(k) < \infty$, then $u(k) = 2^n$ for some integer $n \geq 0$ \cite{kaplansky}. This conjecture was disproved by Merkurjev, who first constructed a field $k$ with $u(k) = 6$ \cite{mer}, and later showed that for each integer $n \geq 1$ there is a field $k_n$ with $u(k_n) = 2n$ \cite{merkurjev}. The question still remained if there is a field with odd $u$-invariant larger than~$1$. Izhboldin answered this question affirmatively by constructing a field with $u$-invariant 9 \cite{izh01}, and this result was later generalized by Vishik, who showed that for each integer $r \geq 3$ there is a field with $u$-invariant $2^r + 1$ \cite{vishik}. 

Much less is known about the possible values of the $m$-invariant. Like the $u$-invariant, for each integer $n \geq 0$ one can easily find fields with $m$-invariant $2^n$ (for example, ${m(\mathbb{C}) = 1}$ and $m(\mathbb{C}((t_1))((t_2)) \cdots ((t_n))) = 2^n$), and it is rather straightforward to show that $m(k) \ne 3$ or $5$ for any field $k$ of characteristic~${\ne 2}$ \cite[p.194, 1.1a)]{m-inv}. Furthermore, if $k$ is \textit{linked} (i.e., if the quaternion algebras over~$k$ form a subgroup of the Brauer group $\Br (k)$) then $m(k) \ne 7$ \cite[p.195, 1.1b)]{m-inv}. For all fields $k$ of characteristic $\ne 2$ we also know that $m(k) \leq u(k)$, and if $m(k) = u(k) < \infty$, then $m(k) = u(k) = 2^n$ for some integer $n \geq 0$ \cite{universal, m-inv}. So if $u(k)$ is finite and $m(k) \ne 2^n$ for any integer $n \geq 0$, then $m(k) < u(k)$. 

However, very few examples of fields whose $m$-invariant is not a power of two exist in the literature. To the author's knowledge, the first example of such a field was found by Hoffmann, who constructed a linked field $k$ with $m(k) = 6$ \cite[Proposition~4.3]{6-dim}. In \cite[Proposition~3.40]{refinements}, the author constructed a non-linked field with $m$-invariant 6. The goal of this article is to construct a number of other fields~$F$ with $m(F) \ne 2^n$ for any integer $n \geq 0$. The main result is the following (see Theorems \ref{any even m-inv}, \ref{odd m-inv 1}, and \ref{odd m-inv 2}).

\begin{maintheoremintro*}
For each positive integer $n \ne 3, 5$ there is a field $F$ of characteristic $\ne 2$ with $m(F) = n$.
\end{maintheoremintro*}

The structure of this article is as follows. We begin by recalling necessary background and introducing notation in Section \ref{background}. In Section \ref{preliminaries} we prove preliminary results about the isotropy behavior of quadratic forms over function fields of quadrics that will be useful in our constructions. In Section \ref{main construction} we describe the construction we use to build fields with prescribed $m$-invariants. In Section \ref{proof of the main theorem} we use this construction to prove the main theorem, showing that each positive integer different from $3$ and $5$ can be realized as the $m$-invariant of a field. Similar constructions have been used to build fields with other prescribed invariants in, e.g., \cite{becher, leep, hoff99, hornix, izh01, merkurjev, vishik}. The arguments in Section \ref{proof of the main theorem} (particularly those used when constructing fields of prescribed odd $m$-invariant) use highly technical results in quadratic form theory, so in Section \ref{main results} we give examples of alternative arguments for particular integers that use more concrete results regarding quadratic forms of specific dimensions. 

\section{Background and notation}
\label{background}
All of the fields considered will have characteristic different from $2$, and all quadratic forms (occasionally referred to just as forms) considered will be non-degenerate (or \textit{regular}). Our notation and terminology follows \cite{lam}, and we assume familiarity with basic notions of quadratic form theory (see, e.g., \cite[Chapter~I]{lam}). 

Let $q$ be a quadratic form over a field $k$ of characteristic $\ne 2$. We say that $q$ is \textit{universal} over $k$ if $q$ represents all non-zero elements of $k$; i.e., for each $a \in k^{\times}$ there is some $x$ such that $q(x) = a$. The \textit{Witt index} of $q$ will be denoted by $i_W(q)$, the \textit{anisotropic part} of $q$ will be denoted by~$q_{an}$, and $c(q)$ will denote the \textit{Witt invariant} of $q$, i.e., the Brauer class of the Clifford algebra~$C(q)$ of~$q$ if $\dim q$ is even or the Brauer class of the even part $C_0(q)$ of the Clifford algebra of $q$ if $\dim q$ is odd. The \textit{Schur index} of $q$ will be denoted by $\ind q$ and is defined to be the (Schur) index of $C_0(q)$ if $\dim q$ is odd or the (Schur) index of~$C(q)$ if $\dim q$ is even. For elements $a, b \in k^{\times}$, the quaternion algebra over $k$ with generators~$i, j$ such that $i^2 = a$, $j^2 = b$, and $ij = -ji$ will be denoted by $\left(\frac{a, b}{k}\right)$. For positive integers $n \geq 1$, $I^n(k)$ will denote the $n$-th power of the fundamental ideal $I(k)$ of even-dimensional quadratic forms in the Witt ring $W(k)$. For any field extension $K/k$,~$q_K$ will denote the form $q$ considered as a quadratic form over $K$.

Given any $a_1, \ldots, a_n \in k^{\times}$, the \textit{$n$-fold Pfister form} $\langle 1, a_1 \rangle \otimes \cdots \otimes \langle 1, a_n \rangle$ over $k$ will be denoted by $\langle \langle a_1, \ldots, a_n \rangle \rangle$. We say that a quadratic form $q$ over $k$ is a \textit{Pfister neighbor} if $(\langle a \rangle \otimes q) \perp p \simeq \varphi$ for some $a \in k^{\times}$, some quadratic form $p$ over $k$, and some Pfister form~$\varphi$ over $k$ with $\dim \varphi < 2 \dim q$. The symbol $\simeq$ denotes isometry of quadratic forms and $\perp$ denotes the orthogonal sum of two quadratic forms. The \textit{hyperbolic plane} $x^2 - y^2$ will be denoted by $\mathbb{H}$, and for any positive integer $r \geq 1$, $r\mathbb{H}$ denotes the orthogonal sum of $r$ copies of $\mathbb{H}$.

Our main construction utilizes function fields of quadratic forms. For a quadratic form~$q$ over a field $k$ with $\dim q \geq 2$ and $q \not\simeq \mathbb{H}$, the \textit{function field} of $q$, denoted by $k(q)$, is the function field of the projective variety defined by $q = 0$. The form $q$ is isotropic over~$k(q)$, $k(q)$ has transcendence degree $\dim q - 2$ over $k$, and $k(q)$ is a purely transcendental extension of~$k$ if and only if~$q$ is isotropic over~$k$ (see, e.g., \cite[p.463]{function fields}). For a family $\{q_{\alpha}\}$ of quadratic forms over~$k$ with $\dim q_{\alpha} \geq 3$ for all $\alpha$, the field $k(\{q_{\alpha}\})$ is the free compositum of the fields~$k(q_{\alpha})$ for all~$\alpha$ \cite[p.333]{lam}. 

At certain points of this article we will appeal to the \textit{splitting pattern} of a non-hyperbolic quadratic form~$q$ over $k$, whose definition we now recall. Let $q_0 = q_{an}$ and $k_0 = k$. For $j \geq 1$, inductively define $k_j = k_{j-1}(q_{j-1})$ to be the function field of~$q_{j-1}$ and let $q_j = \left((q_{j-1})_{k_j}\right)_{an}$, stopping when $\dim q_h \leq 1$. Since $\dim q_j < \dim q_{j-1}$, this process terminates after a finite number of steps, and we call this number $h$ the \textit{height} of $q$. Letting $i_j(q) = i_W((q_{j-1})_{k_j})$, the splitting pattern of $q$ is then given by $(i_1(q), \ldots, i_h(q))$. (Some authors instead call this sequence the \textit{incremental splitting pattern}, and they use the term ``splitting pattern" to refer to the sequence $(i_1(q), i_1(q) + i_2(q), \ldots, i_1(q) + \ldots + i_h(q))$.) Of particular interest is the \textit{first Witt index} $i_1(q)$ of an anisotropic quadratic form $q$ with $\dim q \geq 2$ over~$k$, i.e., the Witt index of~$q_{k(q)}$. 

\section{Preliminary results}
\label{preliminaries}
In this section we prove preliminary results about the isotropy behavior of quadratic forms over the function fields of certain quadratic forms. We begin with the following well-known result, which we use frequently throughout the remainder of the article.
\begin{lemma}
\label{anisotropic over purely transcendental}
Let $q$ be an anisotropic quadratic form over a field $L$ of characteristic $\ne 2$. Then $q$ remains anisotropic over all purely transcendental extensions of $L$.
\end{lemma}
\begin{proof}
See, e.g., \cite[Lemma~IX.1.1]{lam} or \cite[Lemma~7.15]{ekm}.
\end{proof}

The next two results will help us prove lower bounds for the $m$-invariants of the fields we construct. 
\begin{lemma}
\label{Witt index vs. first Witt index}
Let $q$ be an anisotropic quadratic form of dimension at least two over a field~$L$ of characteristic $\ne 2$. Then for each $a \in L^{\times}$ such that $q \perp \langle -a \rangle$ is anisotropic, letting $L' = L(q \perp \langle -a \rangle)$ we have $i_W(q_{L'}) = i_1(q \perp \langle -a \rangle) - 1$.
\end{lemma}
\begin{proof}
We first observe that \cite[Exercise I.16(2)]{lam} implies $i_W(q_{L'}) \geq i_1(q \perp \langle -a \rangle) - 1$. To complete the proof we must prove the opposite inequality.

First, suppose that $q_{L'}$ is anisotropic. Then $i_W(q_{L'}) = 0$, and since $i_1(p) \geq 1$ for any anisotropic quadratic form $p$ over $L$ with $\dim p \geq 2$, we have $i_1(q \perp \langle -a \rangle) \geq 1 = i_W(q_{L'}) + 1$, proving the desired inequality in this case.

Next, suppose that $q_{L'}$ is isotropic. Since $q$ and $q \perp \langle -a \rangle$ are anisotropic over $L$ by assumption, by \cite[Corollary~4.2]{karpenko-merkurjev} (see also \cite[Theorem~1.1]{scully}) we have
\[
	i_W(q_{L'}) \leq \dim q - \dim(q \perp \langle -a \rangle) + i_1(q \perp \langle -a \rangle) = i_1(q \perp \langle -a \rangle) - 1,
\]
as desired.
\end{proof}

\begin{prop}
\label{remains anisotropic}
Let $q$ be an anisotropic quadratic form of dimension at least two over a field~$L$ of characteristic $\ne 2$, and let $a \in L^{\times}$ be any element such that $q$ remains anisotropic over ${L' = L(q \perp \langle -a \rangle)}$. Then all anisotropic quadratic forms $p$ over $L$ with ${1 \leq \dim p \leq \dim q}$ remain anisotropic over $L'$.
\end{prop}
\begin{proof}
The claim is trivial if $\dim p = 1$, so suppose $\dim p \geq 2$. If $q \perp \langle -a \rangle$ is isotropic over~$L$, then~$L'$ is a purely transcendental extension of $L$, so $p$ remains anisotropic over $L'$ by Lemma~\ref{anisotropic over purely transcendental}. If $q \perp \langle -a \rangle$ is anisotropic over $L$, then since~$q$ remains anisotropic over $L'$ by assumption, Lemma~\ref{Witt index vs. first Witt index} implies $i_1(q \perp \langle -a \rangle) = 1$. Therefore \cite[Corollary~4.2]{karpenko-merkurjev} implies all anisotropic quadratic forms over $L$ of dimension $\leq \dim (q \perp \langle -a \rangle) - i_1(q \perp \langle -a \rangle) = \dim q$ remain anisotropic over $L'$.
\end{proof}

\section{The main construction}
\label{main construction}
In this section we describe the construction we will use to build fields $F$ with prescribed $m$-invariant. This is similar to the construction found in \cite[Proposition~4.3]{6-dim}, and is a common approach to building fields with desired invariants (see, e.g., \cite{becher, leep, refinements, 6-dim, hoff99, hornix, izh01, merkurjev, vishik}).
\begin{const}
\label{construction}
Let $k$ be a field of characteristic $\ne 2$ over which there exists an anisotropic quadratic form $\psi$ with $\dim \psi \geq 2$. Define fields $E_i$ for $i \geq 0$ inductively as follows. Let $\widetilde{E}_i = E_i(X_i)$ for an indeterminate~$X_i$. Let
\begin{align*}
	E_0 &= k \\
	E_i &= \widetilde{E}_{i-1}\left(\{\psi \perp \langle -a \rangle \mid a \in E_{i-1}^{\times}\}\right) \text{ for } i \geq 1.
\end{align*}
Let $F = \bigcup_{i = 0}^{\infty} E_i$.
\end{const}

\begin{prop}
\label{possible m-inv}
Let $k, \psi, E_i, F$ be as in Construction \ref{construction}. If $\psi$ remains anisotropic over~$F$, then $m(F) = \dim \psi$.
\end{prop}
\begin{proof}
Let $a \in F^{\times}$ be arbitrary. Then $a \in E_i^{\times}$ for some $i \geq 0$. Hence $\psi \perp \langle -a \rangle$ is isotropic over~$E_{i+1}$ and thus over $F$. Therefore the anisotropic form $\psi$ represents $a$ over $F$, and since $a \in F^{\times}$ was arbitrary, we conclude that $\psi$ is universal over $F$. Thus $m(F) \leq \dim \psi$.

To prove the opposite inequality, we will show that there are no anisotropic universal quadratic forms over $F$ with dimension $< \dim \psi$. To that end, let $q$ be any anisotropic quadratic form over~$F$ with $\dim q < \dim \psi$. The form $q$ is defined (and anisotropic) over~$E_i$ for some $i \geq 0$. One can then easily show that the form $q_i := q_{\widetilde{E}_i} \perp \langle -X_i \rangle$ is anisotropic over~$\widetilde{E}_i$. We will show that~$q_i$ is anisotropic over~$E_j$ for all $j > i$, which implies $q_i$ is anisotropic over $F$. This will complete the proof since~$q_i$ being anisotropic over $F$ implies that $q$ does not represent $X_i \in F^{\times}$, hence $q$ is not universal over~$F$.

By assumption, $\psi$ remains anisotropic over $F$, so $\psi$ is necessarily anisotropic over $E_n$ for all $n \geq 0$. Since $q_i$ is anisotropic over $\widetilde{E}_i$, if $q_i$ became isotropic over $E_{i+1}$ then $q_i$ would necessarily be isotropic over the function field of $\psi \perp \langle -a_i \rangle$ for some $a_i \in E_i^{\times}$. However,~$\psi$ is anisotropic over $E_{i+1}$, hence is anisotropic over the function field of $\psi \perp \langle -a \rangle$ for all $a \in E_i^{\times}$. So, since $\dim q_i = \dim q + 1 \leq \dim \psi$, Proposition \ref{remains anisotropic} implies that $q_i$ remains anisotropic over~$E_{i+1}$. This, in turn, implies $q_i$ is anisotropic over $\widetilde{E}_{i+1}$ by Lemma \ref{anisotropic over purely transcendental}. Iterating the above argument shows that~$q_i$ is anisotropic over $E_j$ for all $j > i$, and the proof is complete.
\end{proof}

\section{Proof of the main theorem}
\label{proof of the main theorem}
As we saw in the introduction, using rather elementary methods one can show that there are no fields with $m$-invariant $3$ or~$5$. However, in this section we will use Construction \ref{construction} to show that $3$ and $5$ are the only positive integers that cannot be realized as the $m$-invariant of a field. We separate this investigation into three cases (Theorems \ref{any even m-inv}, \ref{odd m-inv 1}, and \ref{odd m-inv 2}), as the methods of proof differ slightly. We begin with positive even integers.

\begin{theorem}
\label{any even m-inv}
For each positive integer $n \geq 1$ there is a field $F$ of characteristic $\ne 2$ with $m(F) = 2n$.
\end{theorem}
\begin{proof}
Let $k$ be a field of characteristic $\ne 2$ over which there exists a central division algebra~$D$ that is the product of $n-1$ quaternion algebras. For example, we may take $k = \mathbb{Q}(t_1, \ldots, t_{n-1}, u_1, \ldots, u_{n-1})$ and $D = \left(\frac{t_1, u_1}{k}\right) \otimes_k \left(\frac{t_2, u_2}{k}\right) \otimes_k \cdots \otimes_k \left(\frac{t_{n-1}, u_{n-1}}{k}\right)$ (see, e.g., the first paragraph of the proof of \cite[Theorem 4]{merkurjev}). Then $D$ has index $2^{n-1}$ and by \cite[Lemma~38.3]{ekm} there is a quadratic form $\psi \in I^2 (k)$ with $\dim \psi = 2n$ and $c(\psi) = [D]$ in $\Br (k)$, where $[D]$ is the class of $D$ in $\Br (k)$. The quadratic form $\psi$ is anisotropic over $k$ by \cite[Corollary~38.2]{ekm}.

Now let $E_i, F$ be the fields constructed via Construction \ref{construction} using $k, \psi$. If $B$ is a division algebra of index $< 2^n$ over a field $L$ for some integer $n \geq 1$, then if $L'$ is either a purely transcendental extension of $L$ or the function field of a regular quadratic form $\varphi$ over $L$ with $\dim \varphi \geq 2n + 1$, then $B_{L'}:= B \otimes_L L'$ is a division algebra over $L'$. This follows from, e.g., \cite[p.116]{ekm} in the former case and from \cite[Corollary~30.9]{ekm} in the latter case. Applying this argument repeatedly to the division algebra $D$ over $k$, we conclude that $D_{E_i}$ is a division algebra over $E_i$ for all $i \geq 0$, hence a division algebra over $F$. Now applying \cite[Corollary~38.2]{ekm}, since $c(\psi_F) = [D_F]$, we conclude that $\psi$ remains anisotropic over~$F$, hence $m(F) = \dim \psi = 2n$ by Proposition \ref{possible m-inv} and the proof is complete.
\end{proof}

We now shift focus to realizing all positive odd integers different from $3$ and $5$ as the $m$-invariant of a field. We begin with positive odd integers that cannot be written as $2^r + 1$ for any integer $r$.
\begin{theorem}
\label{odd m-inv 1}
For each positive odd integer $n$ that is not of the form $2^r + 1$ with $r$ a positive integer, there is a field $F$ of characteristic $\ne 2$ with $m(F) = n$.
\end{theorem}

The key ingredient to the proof of Theorem \ref{odd m-inv 1} is the following result.
\begin{lemma}
\label{splitting pattern}
Let $q$ be an odd-dimensional anisotropic quadratic form over a field $L$ of characteristic $\ne 2$ with $\dim q = 2s + 1$ for some integer $s \geq 3$ such that $s \ne 2^r$ for any integer $r \geq 2$. Suppose $q$ has trivial discriminant and \emph{ind} $q = 2^{s-1}$. Then for all $a \in L^{\times}$, $q$ remains anisotropic over $L' = L(q \perp \langle -a \rangle)$.
\end{lemma}
\begin{proof}
Let $a \in L^{\times}$ be given. If $q \perp \langle -a \rangle$ is isotropic, then $L'$ is a purely transcendental extension of $L$ and the claim follows from Lemma \ref{anisotropic over purely transcendental}. So to complete the proof it suffices to consider the case when $q \perp \langle -a \rangle$ is anisotropic. Before proceeding, we observe that, because we assumed $q$ has trivial discriminant, a straightforward computation shows that $c(q \perp \langle -a \rangle) = c(q)$, hence $\ind (q \perp \langle -a \rangle) = \ind q = 2^{s-1}$. 

First, suppose that $q \perp \langle -a \rangle \in I^2(L)$. Then because $\ind (q \perp \langle -a \rangle) = 2^{s-1}$ and ${\dim (q \perp \langle -a \rangle) = 2s + 2}$, \cite[Proposition~2.1]{hoff98} implies that $i_1(q \perp \langle -a \rangle) = 1$. Now applying Lemma \ref{Witt index vs. first Witt index}, we conclude that $q_{L'}$ is anisotropic.

Finally, suppose that $q \perp \langle -a \rangle \not\in I^2(L)$. Then since $\ind(q \perp \langle -a \rangle) = 2^{s-1}$, it follows from \cite[Proposition~2.1]{hoff98} that $i_1(q \perp \langle -a \rangle) \leq 2$. Consequently, Lemma \ref{Witt index vs. first Witt index} implies that ${i_1(q \perp \langle -a \rangle) = 2}$ if and only if $q_{L'}$ is isotropic. By contradiction, suppose $q_{L'}$ is isotropic. Then $i_1(q \perp \langle -a \rangle) = 2$ and \cite[Proposition~2.1]{hoff98} implies that $q \perp \langle -a \rangle$ has (incremental) splitting pattern $(2, 1, 1, \ldots, 1)$. It then follows from \cite[Theorem~7.7(a), Theorem~7.8, Theorem~4.20]{motives} that $\dim (q \perp \langle -a \rangle) = 2^r + 2$ for some integer $r$, hence $\dim q = 2^r + 1$ for some $r$. We have reached a contradiction, thus $q_{L'}$ is anisotropic, as desired.
\end{proof}

\begin{proof}[Proof of Theorem \ref{odd m-inv 1}]
Let $n \geq 1$ be given. If $n = 1$, then letting $F$ be any algebraically closed field of characteristic $\ne 2$, we have $m(F) = 1$, proving the theorem in this case. Next, suppose $n > 1$ and write $n = 2s + 1$ for some integer $s \geq 3$ such that $s \ne 2^r$ for any integer $r \geq 2$. Let $k$ be a field of characteristic $\ne 2$ over which there exists an anisotropic quadratic form $\psi$ with $\dim \psi = n = 2s + 1$ and $\ind \psi = 2^{s-1}$. For example, we may take $k = k_0(t)$ where $k_0$ is a field of characteristic $\ne 2$ over which there exists a central division algebra~$D$ that is the product of $s-1$ quaternion algebras and let $\psi = \varphi_k \perp \langle t \rangle$ for $\varphi \in I^2(k_0)$ a $2s$-dimensional form over $k_0$ with $c(\varphi) = [D] \in \Br(k_0)$. After scaling $\psi$ if necessary (which does not affect $\ind \psi$ since $\dim \psi$ is odd), we may assume that $\psi$ has trivial discriminant. 

Now let $E_i, F$ be the fields constructed via Construction \ref{construction} using $k, \psi$. We will inductively show that $\psi$ remains anisotropic over $E_i$ with $\ind \psi_{E_i} = 2^{s-1}$ for all $i \geq 0$. This will, in turn, imply that $\psi$ remains anisotropic over $F$, hence $m(F) = \dim \psi = n$ by Proposition~\ref{possible m-inv}. The base case of $i = 0$ is trivial since $E_0 = k$. Now suppose there is some $i \geq 0$ such that~$\psi_{E_i}$ is anisotropic with $\ind \psi_{E_i} = 2^{s-1}$. The Schur index of $\psi_{E_i}$ remains unchanged over purely transcendental extensions (see, e.g., \cite[Lemma~1.2]{izh01}), thus $\ind \psi_{\widetilde{E}_i} = 2^{s-1}$ and $\psi_{\widetilde{E}_i}$ is anisotropic by Lemma \ref{anisotropic over purely transcendental}. Now suppose $\widetilde{E}_i \subseteq L \subseteq E_{i+1}$ is any intermediate field extension over which $\ind \psi_L = 2^{s-1}$ and $\psi_L$ is anisotropic. Then for any $a \in E_i^{\times}$, \cite[Corollary~30.9]{ekm} implies $\ind \psi_{L(\psi \perp \langle -a \rangle)} = 2^{s-1}$, and Lemma \ref{splitting pattern} implies $\psi_{L(\psi \perp \langle -a \rangle)}$ is anisotropic. We therefore conclude that $\psi_{E_{i+1}}$ is anisotropic with $\ind \psi_{E_{i+1}} = 2^{s-1}$, completing the proof by induction, and the proof of the theorem as a whole.
\end{proof}
To conclude this section we complete the proof of the Main Theorem by realizing the remaining odd integers $2^r + 1$ for positive integers $r \geq 3$ as the $m$-invariant of a field.
\begin{theorem}
\label{odd m-inv 2}
For each positive integer $r \geq 3$ there is a field $F$ of characteristic $\ne 2$ with $m(F) = 2^r + 1$.
\end{theorem}
\begin{proof}
Let $k_0$ be any field of characteristic $0$, let $k = k_0(t_1, t_2, \ldots, t_{2^r + 1})$, and consider the generic anisotropic $(2^r + 1)$-dimensional quadratic form $\psi = \langle t_1, t_2, \ldots, t_{2^r + 1} \rangle$ over $k$. In \cite{vishik}, Vishik showed that the anisotropy of $\psi$ over field extensions $L/k$ is detected by a certain property $B(L)$ of the \textit{elementary discrete invariant} of $\Psi_L$, denoted by EDI$(\Psi_L)$, where $\Psi$ is the projective quadric defined by $\psi = 0$. That is, if $B(L)$ is satisfied then $\psi_L$ is anisotropic. See \cite[Section 2]{vishik} for the definition of EDI$(\Psi)$ and \cite[Theorem~5.1, Corollary~5.2]{vishik} for the definition of $B(L)$. If $L/k$ is a field extension over which $B(L)$ is satisfied and $L'$ is either a purely transcendental extension of $L$ or the function field of a quadratic form over $L$ of dimension $> \dim \psi$, then $B(L')$ is satisfied (see \cite[Theorem~5.1]{vishik} for the function field case). In particular, if $B(L)$ is satisfied then $B(L(\psi \perp \langle -a \rangle))$ is satisfied for all $a \in L^{\times}$. Furthermore, for a direct system $\{L_j\}_j$ of field extensions of $k$ with limit $L_{\infty}$, if $B(L_j)$ is satisfied for all $j$, then $B(L_{\infty})$ is satisfied (see the proof of \cite[Corollary~5.2]{vishik}).  

Now let $E_i, F$ be the fields constructed via Construction \ref{construction} using $k, \psi$. Because the form~$\psi$ is generic over $k$, it follows that $B(k)$ is satisfied (see the proof of \cite[Corollary~5.2]{vishik}). This then implies that $B(E_i)$ is satisfied for all $i \geq 0$, hence $B(F)$ is satisfied. Therefore $\psi$ is anisotropic over $F$ and we conclude that $m(F) = \dim \psi = 2^r + 1$ by Proposition \ref{possible m-inv}.
\end{proof}

\section{Examples of alternative arguments}
\label{main results}
In Section \ref{proof of the main theorem} we showed that for each positive integer $n \ne 3, 5$ there is a field of characteristic $\ne 2$ with $m$-invariant $n$. The arguments in that section, particularly those used to construct fields of prescribed odd $m$-invariant, relied on highly technical results in quadratic form theory like the Binary Motive Theorem \cite[Theorem~4.20]{motives}. In this section we will give examples of alternative arguments for the existence of fields with prescribed $m$-invariants that use more concrete results, most of which are about low-dimensional quadratic forms and can be understood in more elementary terms than, say, the Binary Motive Theorem or \cite[Theorem~5.1]{vishik}.

\begin{example}[Field with $m$-invariant $7$]
\label{m-inv 7}
Let $k$ be a field of characteristic $\ne 2$ over which there is an anisotropic non-universal Albert form~$\varphi$ (i.e., $\dim \varphi =$~$6$ and $\varphi \in I^2(k)$). For example, take $k = \mathbb{C}((t_1))((t_2))((t_3))((t_4))$. Recall that for each field extension $K/k$, $\ind \varphi_K = 4$ if and only if $\varphi$ is anisotropic over~$K$ \cite[Theorem~3.12(iii)]{jacobson}. After scaling~$\varphi$ if necessary, we may assume that the non-universal form~$\varphi$ does not represent $1$ over $k$, hence the form $\psi := \varphi \perp \langle -1 \rangle$ is anisotropic over~$k$. A straightforward calculation shows that $c(\varphi) = c(\psi)$, and since~$\varphi$ is anisotropic over $k$, we conclude $\ind \psi = \ind \varphi =$~$4$.

Now let $E_i, F$ be the fields constructed via Construction \ref{construction} using $k, \psi$. Since anisotropic Albert forms remain anisotropic over the function field of any quadratic form of dimension~${\geq 7}$ \cite[Lemma~1.1, Proposition~5.6]{mer}, \cite[Theorem XIII.2.6]{lam} and remain anisotropic over purely transcendental extensions by Lemma \ref{anisotropic over purely transcendental}, we conclude that the anisotropic Albert form $\varphi$ over $k$ remains anisotropic over $F$. Now consider any intermediate field extension $k \subseteq L \subseteq F$. The form $\varphi$ being anisotropic over $F$ has a number of consequences for such fields $L$. First, this implies that $\varphi$ is anisotropic over~$L$, hence $\ind \psi_L = \ind \varphi_L = 4$. Next,~$\varphi$ is not a Pfister neighbor over $L$ since Albert forms are Pfister neighbors if and only if they are hyperbolic. This, in turn, implies that for all $a \in L^{\times}$, $\psi \perp \langle -a \rangle$ is not a Pfister neighbor over $L$. As a consequence, if $\psi$ is anisotropic over~$L$, then since $\ind \psi_L = 4$ and $\psi \perp \langle -a \rangle$ is not a Pfister neighbor for all $a \in L^{\times}$, \cite[Proposition~3.10(1)]{kar04} implies that $\psi$ remains anisotropic over $L(\psi \perp \langle -a \rangle)$ for all $a \in L^{\times}$. This argument together with Lemma~\ref{anisotropic over purely transcendental} shows that~$\psi$ is anisotropic over $E_i$ for all $i \geq 0$, hence anisotropic over~$F$. Thus $m(F) = 7$ by Proposition \ref{possible m-inv}.
\end{example}
\begin{remark}
The argument in Example \ref{m-inv 7} shows that the field $F$ with $m(F) = 7$ we constructed is not linked since the anisotropic Albert form $\varphi$ over $k$ remains anisotropic over~$F$. The existence of an anisotropic Albert form over $F$ is necessary since linked fields cannot have $m$-invariant $7$ \cite[p.195, 1.1b)]{m-inv}.
\end{remark}

\begin{example}[Field with $m$-invariant $9$]
\label{m-inv 9}
Let $k$ be a field of characteristic $\ne 2$ over which there exists an essential $9$-dimensional quadratic form $\psi$ (see \cite[Definition~0.2]{izh01} for the definition). For example, we may take $k =$~$k_0(t_1, \ldots, t_9)$ for any field $k_0$ of characteristic $\ne 2$. In general, if $\rho$ is an essential $9$-dimensional quadratic form over a field $L$, then Izhboldin showed that if $L'$ is either a purely transcendental extension of $L$ or the function field of a quadratic form of dimension $\geq 10$ over $L$, then $\rho_{L'}$ is an essential $9$-dimensional quadratic form (see \cite[Lemma~1.24]{izh01} for the former case and \cite[Proposition~0.4]{izh01} for the latter case). Now let $E_i, F$ be the fields constructed via Construction \ref{construction} using $k, \psi$. Then since $\psi$ is essential over $k$, by an inductive argument that uses these results of Izhboldin, one can show that $\psi_{E_i}$ is an essential $9$-dimensional quadratic form for all $i \geq 0$. Essential $9$-dimensional quadratic forms are anisotropic by definition, so $\psi_{E_i}$ is anisotropic for all $i \geq 0$. Thus $\psi$ is anisotropic over $F$, so $m(F) = 9$ by Proposition \ref{possible m-inv}.
\end{example}

In Example \ref{m-inv 13} we will construct a field with $m$-invariant 13, using the following lemma in the process.

\begin{lemma}
\label{13-dimensional lemma}
Let $q$ be an anisotropic $13$-dimensional quadratic form over a field $L$ of characteristic $\ne 2$ such that $\det q = 1$, $\emph{ind } q = 4$, and $q$ has splitting pattern $(1, 3, 1, 1)$. Let~$L'$ be either a purely transcendental extension of $L$ or the function field of $q \perp \langle -a \rangle$ for any $a \in L^{\times}$. Then 
\begin{enumerate}[label=(\alph*)]
	\item $\emph{ind }q_{L'} = 4$, 
	
	\item $q_{L'}$ is anisotropic,
	
	\item $q_{L'}$ has splitting pattern $(1, 3, 1, 1)$.
\end{enumerate}
\end{lemma}
\begin{proof}
The Schur index of any quadratic form is preserved under purely transcendental extension (see, e.g., \cite[Lemma~1.2]{izh01}), and since $\ind q = 4$, the Schur index of $q$ also remains unchanged over the function field of any quadratic form of dimension $\geq 7$ by \cite[Corollary~30.9]{ekm}. Since $\dim(q \perp \langle -a \rangle) = 14$ for all $a \in L^{\times}$, this proves (a). 

Before proceeding, we observe that since $i_1(q) = 1 \ne 5$, it follows that $q$ is not a Pfister neighbor \cite[Corollary~2]{function fields}. We now focus on proving (b), i.e., that $q_{L'}$ is anisotropic. If~$L'$ is purely transcendental over $L$, then this follows from Lemma \ref{anisotropic over purely transcendental}. So suppose ${L' = L(q \perp \langle -a \rangle)}$ for some $a \in L^{\times}$ such that $q \perp \langle -a \rangle$ is anisotropic over $L$ and, by contradiction, suppose $q_{L'}$ is isotropic. Then Lemma \ref{Witt index vs. first Witt index} implies $i_1(q \perp \langle -a \rangle) \geq 2$. Because $\dim(q \perp \langle -a \rangle) = 14$ and $i_1(q \perp \langle -a \rangle) \geq 2$, \cite[Conjecture~0.1]{kar03} implies $i_1(q \perp \langle -a \rangle) = 2$ or~$6$. If $i_1(q \perp \langle -a \rangle) = 6$, then $q \perp \langle -a \rangle$ is a Pfister neighbor \cite[Lemma~3.3]{hoff98}, which implies that $q$ is a Pfister neighbor. As we saw at the beginning of this paragraph,~$q$ is not a Pfister neighbor, thus $i_1(q \perp \langle -a \rangle) = 2$. (This also follows from \cite[Corollary~4.9]{motives} since $i_1(q) = 1$.) Now, since $q$ has splitting pattern $(1, 3, 1, 1)$, \cite[Theorem~7.6]{motives} implies that the splitting pattern of $q \perp \langle -a \rangle$ is a ``specialization'' of $(1, 1, 2, 1, 1, 1)$. That is, the splitting pattern of $q \perp \langle -a \rangle$ is obtained by possibly adding some collections of adjacent entries and possibly removing some initial entries from $(1, 1, 2, 1, 1, 1)$. Furthermore, Totaro showed that the only possible splitting patterns of 14-dimensional anisotropic quadratic forms with first Witt index $2$ are $(2, 2, 2, 1)$ and $(2, 1, 1, 2, 1)$ \cite[p.264]{totaro}. The splitting pattern $(2, 1, 1, 2, 1)$ is not a specialization of $(1, 1, 2, 1, 1, 1)$, therefore $q \perp \langle -a \rangle$ must have splitting pattern $(2, 2, 2, 1)$. This, in turn, implies that $q \perp \langle -a \rangle \simeq \langle \langle -a \rangle \rangle \otimes \varphi$ for some $7$-dimensional quadratic form $\varphi$ over $L$ \cite[p.265]{totaro}. A straightforward calculation shows that $c(q \perp \langle -a \rangle) = c(\langle \langle -a \rangle \rangle \otimes \varphi) = \left(\frac{a, -\det \varphi}{L}\right)$. Now, since $\det q = 1$, it follows that $c(q) = c(q \perp \langle -a \rangle)$, thus $c(q) = \left(\frac{a, -\det \varphi}{L}\right)$. This implies $\ind q \leq 2 < 4$, which is a contradiction of our assumption on $q$. Therefore $q_{L'}$ is anisotropic, proving (b). 

By \cite[Theorem~7.5]{motives}, the splitting pattern of $q_{L'}$ is a specialization of the splitting pattern of $q$. Furthermore, Vishik showed that the only possible splitting patterns of anisotropic $13$-dimensional quadratic forms are $(5, 1)$, $(1, 3, 1, 1)$, $(1, 1, 1, 3)$, and $(1, 1, 1, 1, 1, 1)$ \cite[Table~1, p.84]{motives}. Since $q_{L'}$ is anisotropic and the splitting pattern of $q_{L'}$ must be a specialization of $(1, 3, 1, 1)$, it follows that the splitting pattern of $q_{L'}$ is either $(5, 1)$ or $(1, 3, 1, 1)$. If $q_{L'}$ had splitting pattern $(5, 1)$, then $q_{L'}$ would be a Pfister neighbor \cite[Lemma~3.3]{hoff98}. However, because $q$ is not a Pfister neighbor over $L$, $q_{L'}$ is not a Pfister neighbor over~$L'$. This follows from \cite[Proposition~7]{function fields} if $L'$ is purely transcendental over $L$ and follows from \cite[Lemma~5, Corollary~2]{function fields} if $L'$ is the function field of $q \perp \langle -a \rangle$ for some $a \in L^{\times}$. Therefore $q_{L'}$ has splitting pattern $(1, 3, 1, 1)$, proving (c).
\end{proof}
\begin{example}[Field with $m$-invariant $13$]
\label{m-inv 13}
Let $k$ be a field of characteristic $\ne 2$ over which there exists an anisotropic $13$-dimensional quadratic form $\psi$ with $\det \psi = 1$, $\ind \psi = 4$, and splitting pattern $(1,3,1,1)$. See, e.g., \cite[p.79]{motives} for an example of such a field~$k$. Now let $E_i, F$ be the fields constructed via Construction \ref{construction} using $k, \psi$. Then by an inductive argument, one can use Lemma \ref{13-dimensional lemma} to show that $\det \psi_{E_i} = 1$, $\ind \psi_{E_i} = 4$, and $\psi_{E_i}$ is anisotropic with splitting pattern $(1, 3, 1, 1)$ for all $i \geq 0$. Since $\psi_{E_i}$ is anisotropic for all $i \geq 0$, it follows that $\psi$ remains anisotropic over $F$. Thus $m(F) = 13$ by Proposition~\ref{possible m-inv}.
\end{example}

\begin{remark}
The choice of $\psi$ in Example \ref{m-inv 13} is different from the choice of $\psi$ in the proof of Theorem \ref{odd m-inv 1}. Indeed, in Example \ref{m-inv 13}, $\ind \psi = 4$ and $\psi$ has splitting pattern $(1, 3, 1, 1)$, while in the proof of Theorem \ref{odd m-inv 1}, $\ind \psi = 32$ and $\psi$ has splitting pattern $(1,1,1,1,1,1)$.
\end{remark}

\begin{example}[Field with $m$-invariant $10$]
\label{m-inv 10}
Let $k$ be a field of characteristic $\ne 2$ over which there exists an anisotropic three-fold Pfister form $\pi$ and an anisotropic two-fold Pfister form~$\tau$ such that $\psi := (\pi \perp -\tau)_{an}$ has dimension~$10$. For example, take $k = k_0(t_1, \ldots, t_5)$ for any field $k_0$ of characteristic $\ne 2$, $\pi = \langle \langle t_1, t_2, t_3 \rangle \rangle$, and $\tau = \langle \langle t_4, t_5 \rangle \rangle$. Let $E_i, F$ be the fields constructed via Construction \ref{construction} using $k, \psi$. We will prove by induction that $\psi$ is anisotropic over $E_i$ for all $i \geq 0$. This will imply that $\psi$ is anisotropic over $F$, hence $m(F) = 10$ by Proposition \ref{possible m-inv}. 

Before proceeding, we make two observations about $\psi$. First, $\psi \in I^2(k)$, thus $\psi_L \in I^2(L)$ for all intermediate field extensions $k \subseteq L \subseteq F$. Next, over $k$, $\ind \psi = \ind \tau = 2$ since $c(\psi) = c(\tau)$ and~$\tau$ is an anisotropic 2-fold Pfister form over $k$. By Lemma \ref{anisotropic over purely transcendental} and \cite[Theorem 1]{function fields},~$\tau$ remains anisotropic over purely transcendental extensions and over function fields of quadratic forms of dimension $\geq 5$, respectively. Therefore $\tau$ is anisotropic over $F$, hence $\ind \psi_F = \ind \tau_F = 2$. This, in turn, implies that $\ind \psi_L = 2$ for all intermediate field extensions $k \subseteq L \subseteq F$.

We now show that $\psi$ is anisotropic over $E_i$ for all $i \geq 0$ by induction. The base case of $i = 0$ is trivial since $E_0 = k$ and $\psi$ is anisotropic over $k$. So suppose $\psi$ is anisotropic over~$E_i$ for some $i \geq 0$. Then $\psi$ is anisotropic over $\widetilde{E}_i$ by Lemma \ref{anisotropic over purely transcendental}. Now let $\widetilde{E}_i \subseteq L \subseteq E_{i+1}$ be any intermediate field extension over which $\psi$ is anisotropic. Then $\psi_L \in I^2(L)$ and $\ind \psi_L = 2$, so for all $a \in E_i^{\times}$, $\psi$ remains anisotropic over $L(\psi \perp \langle -a \rangle)$ by \cite[Theorem~4.1]{izh04}. Since this holds for all $a \in E_i^{\times}$, we conclude that $\psi$ is anisotropic over $E_{i+1}$, completing the proof by induction. 
\end{example}

To conclude this article, in Example \ref{m-invariant from holes} we construct a field with $m$-invariant $2^{n+1} - 2^i$ for each pair of integers $(n, i )$ with $1 \leq i \leq n - 1$. By \cite[Conjecture~1.1]{holes}, these integers $2^{n+1} - 2^i$ are the possible positive dimensions of anisotropic quadratic forms in $I^n(k)$ for $n \geq 2$ that are less than $2^{n+1}$ and not a power of two. The following result is key to this construction.
\begin{lemma}
\label{forms in I^n over function fields}
Let $L$ be a field of characteristic $\ne 2$ over which there is an anisotropic quadratic form $q \in I^n(L)$ with $\dim q = 2^{n+1} - 2^i$ for positive integers $1 \leq i \leq n-1$. Then for all $a \in L^{\times}$, $q$ remains anisotropic over $L' = L(q \perp \langle -a \rangle)$.
\end{lemma}
\begin{proof}
If $q \perp \langle -a \rangle$ is isotropic, then $L'$ is purely transcendental over $L$ and the result follows from Lemma \ref{anisotropic over purely transcendental}. So to complete the proof it suffices to consider the case when $q \perp \langle -a \rangle$ is anisotropic. 

We have $\dim (q \perp \langle -a \rangle) - 1 = 2^{n+1} - 2^i = \sum_{j = i}^n 2^j$. By \cite[Conjecture~0.1]{kar03}, it follows that $i_1(q \perp \langle -a \rangle) = 1$ or $1 + \sum_{j = i}^{\ell} 2^j$ for some $i \leq \ell \leq n-1$. Lemma \ref{Witt index vs. first Witt index} then implies that $i_W(q_{L'}) = 0$ or $\sum_{j = i}^{\ell} 2^j$.

By contradiction, suppose $i_W(q_{L'}) = \sum_{j=i}^{\ell} 2^j$. Then we have
\[
	q_{L'} \simeq \left(\sum_{j = i}^{\ell} 2^j\right) \mathbb{H} \perp q'
\]
for some anisotropic quadratic form $q'$ over $L'$ with $\dim q' = 2^{n+1} - \left(2^{\ell + 2} - 2^i\right)$. Since $q \in I^n(L)$, it follows that $q_{L'} \in I^n(L')$ as well, hence $q' \in I^n(L')$. Now, since $q' \in I^n(L')$ is anisotropic with $\dim q' < 2^{n+1}$, we must have $\dim q' = 2^{n+1} - 2^r$ for some $1 \leq r \leq n$ \cite[Conjecture~1.1]{holes}. We saw above that $\dim q' = 2^{n+1} - (2^{\ell + 2} - 2^i)$ for some $1 \leq i \leq \ell \leq n-1$, and one can check that $2^{\ell + 2} - 2^i \neq 2^r$ for any $1 \leq r \leq n$. We have reached a contradiction, and thus conclude that $q$ is anisotropic over $L'$.
\end{proof}

\begin{example}[Field with $m$-invariant $2^{n+1} - 2^i$ for $1 \leq i \leq n - 1$]
\label{m-invariant from holes} Let $(n, i)$ be any pair of integers such that $n \geq 2$ and $1 \leq i \leq n - 1$ and let $k$ be a field of characteristic $\ne 2$ over which there is an anisotropic quadratic form $\psi \in I^n(k)$ with $\dim \psi = 2^{n+1} - 2^i$ (see, e.g., \cite[Section 7]{holes} for such fields~$k$). Let~$F$ be the field constructed via Construction~\ref{construction} using $k, \psi$. For all intermediate field extensions $k \subseteq L \subseteq F$, the form $\psi_L$ belongs to $I^n(L)$. Furthermore, if~$\psi$ is anisotropic over $L$, then $\psi$ remains anisotropic over all purely transcendental extensions of~$L$ by Lemma \ref{anisotropic over purely transcendental}, and $\psi$ remains anisotropic over $L(\psi \perp \langle -a \rangle)$ for all $a \in L^{\times}$ by Lemma~\ref{forms in I^n over function fields}. Applying Lemmas~\ref{anisotropic over purely transcendental} and \ref{forms in I^n over function fields} repeatedly, we conclude that $\psi$ is anisotropic over $F$. Therefore $m(F) = 2^{n+1} - 2^i$ by Proposition \ref{possible m-inv}.
\end{example}

\begin{remark}
The choice of $\psi$ in Examples \ref{m-inv 10} and \ref{m-invariant from holes} differs from the choice of $\psi$ in the proof of Theorem \ref{any even m-inv}. Namely, in Example \ref{m-inv 10}, $\ind \psi = 2$, and in Example \ref{m-invariant from holes}, $\ind \psi = 1$ when $n \geq 3$. However, in the proof of Theorem \ref{any even m-inv}, $\ind \psi = 2^{s-1} > 2$, where $\dim \psi = 2s$.
\end{remark}

\begin{remark}
We do not know the $u$-invariants of the fields $F$ constructed above with $m(F) > 1$. If $n \geq 1$ is the smallest positive integer such that $m(F) \leq 2^n$, then $u(F) \geq 2^n$ by \cite[Corollary~1.6]{m-inv}, but there is the question of whether it is possible to have $u(F) > 2^n$. However, we can slightly alter Construction~\ref{construction} to construct $F$ in a way that guarantees $u(F) = 2^n$, with $n$ as above. Indeed, let $k$ be a field of characteristic $\ne 2$ over which there is an anisotropic quadratic form $\psi$ with $2^{n-1} < \dim \psi \leq 2^n$ for some $n \geq 1$. We inductively define fields $E_j$ for $j \geq 0$ as follows. Let $\widetilde{E}_j = E_j(X_j)$ for an indeterminate $X_j$. Let
\begin{align*}
	E_0 &= k \\ 
	E_{2i+1} &= \widetilde{E}_{2i}\left(\{\tau \in W (E_{2i}) \mid \dim \tau = 2^n + 1\}\right) \text{ for $i \geq 0$} \\
	E_{2i} &= \widetilde{E}_{2i-1}\left(\{\psi \perp \langle -a \rangle \mid a \in E_{2i-1}^{\times}\}\right) \text{ for $i \geq 1$}.
\end{align*}
Let $F = \bigcup_{j = 0}^{\infty} E_j$. Any quadratic form of dimension $2^n + 1$ over $F$ is defined over $E_{2i}$ for some $i \geq 0$, and thus is isotropic over $E_{2i+1}$, hence isotropic over $F$. Therefore $u(F) \leq 2^n$. If~$\rho$ is an anisotropic quadratic form with $\dim \rho \leq 2^n$ over $E_{2i}$ for some $i \geq 0$, then by Lemma~\ref{anisotropic over purely transcendental} and \cite[Theorem~1]{function fields},~$\rho$ remains anisotropic over $E_{2i+1}$. Furthermore, if $\ind \rho < 2^{2^{n-1}}$ over $E_{2i}$ for some $i \geq 0$, then $\ind \rho_{E_{2i+1}} = \ind \rho$ by \cite[Corollary~30.9]{ekm}. So if $k$ and~$\psi$ are selected as in the proofs of Theorems \ref{any even m-inv}, \ref{odd m-inv 1}, and \ref{odd m-inv 2} (as well as in the examples in Section \ref{main results}), then these observations together with the arguments in those proofs show that~$\psi$ is anisotropic over~$E_j$ for all $j \geq 0$, hence anisotropic over~$F$. The form~$\psi$ is universal over $F$ by construction, thus $m(F) \leq \dim \psi$. Moreover, combining the argument in the proof of Proposition \ref{possible m-inv} with the above observation that anisotropic quadratic forms of dimension~$\leq 2^n$ over~$E_{2i}$ remain anisotropic over~$E_{2i+1}$ shows that there are no anisotropic universal quadratic forms over $F$ with dimension $< \dim \psi$. Thus $m(F) = \dim \psi$, which in turn implies that $u(F) \geq 2^n$, hence $u(F) = 2^n$.
\end{remark}

\begin{ack} 
The author would like to thank David Harbater for helpful comments and suggestions and Stephen Scully for helpful discussions that substantially strengthened the above results.
\end{ack}

\providecommand{\bysame}{\leavevmode\hbox to3em{\hrulefill}\thinspace}
\providecommand{\href}[2]{#2}

\medskip

\noindent{\bf Author Information:}\\

\noindent Connor Cassady\\
Department of Mathematics, The Ohio State University, Columbus, OH 43210-1174, USA\\
email: cassady.82@osu.edu
\end{document}